\documentclass[11pt]{article} \usepackage[dvips]{epsfig}
\usepackage{latexsym}
\usepackage{amsfonts,amsmath,amssymb}
\newtheorem{corollary}{Corollary}
\newtheorem{lemma}{Lemma}
\newtheorem{theorem}{Theorem}
\newtheorem{proposition}{Proposition}
\newcommand{\qed}{\mbox{$\Diamond$}\vspace{\baselineskip}}
\newenvironment{proof}{\noindent{\bf Proof:}}{\qed}
\newtheorem{example}{Example}

\begin{document}
\author{Mikl\'os B\'ona\\
        Department of Mathematics\\
University of Florida\\
Gainesville FL 32611-8105\\
USA  }

\title{On Two Related Questions of Wilf Concerning Standard Young Tableaux}

\maketitle

\begin{abstract}
We consider two questions of Wilf related to Standard Young Tableaux.
We provide a partial answer to one question, and that will lead us to 
 a more general answer to the other question.
Our answers are purely combinatorial.
\end{abstract}

\section{Introduction} In 1992, in his paper \cite{wilf}, Herb Wilf has
proved the following interesting result.

\begin{theorem} \label{wilf} (Wilf, \cite{wilf}.) Let $u_k(n)$ be the number
of permutations of length $n$ that contain no increasing subsequence of
length $k+1$, and let $y_k(n)$ be the number of Standard Young Tableaux
on $n$ boxes that have no rows longer than $k$. Then for all even positive
integers $k$, the equality
\begin{equation}
\label{wilfeq} {2n\choose n} u_k(n) = 
\sum_{r=0}^{2n} {2n\choose r} (-1)^r y_k(r)y_k(2n-r)
\end{equation} holds.
\end{theorem}

Wilf's proof of Theorem \ref{wilf} was not elementary; it used modified 
Bessel functions and computed the determinant of a Toeplitz matrix. Therefore,
Wilf asked the following two intriguing questions.

\begin{enumerate}
\item Is there a purely combinatorial proof for Theorem \ref{wilf} ?
\item What statement corresponds to Theorem \ref{wilf}  for odd $k$?
\end{enumerate}

In this paper, we answer Question 1 in a special case, which then will lead
us to a more general answer to Question 2. This answer will be a formula
that will still contain a summation sign, but each summand will be 
non-negative, which will explain why the answer is always non-negative.
The number of nonzero summands will be half of what it is in (\ref{wilfeq}),
 and the summands will be significantly smaller than in (\ref{wilfeq}).

 We point out that in another
special case, that of $k=2$, a simple and elegant bijective proof has 
recently been given by Rebecca Smith and Micah Coleman \cite{smith}. 

We will assume familiarity with the Robinson-Schensted correspondence
between permutations of length $n$ and pairs of Standard Young Tableaux
on $n$ boxes and of the same shape. In particular, we will need the
following facts. 

\begin{enumerate}
\item There is a one-to-one correspondence $RS$ between involutions on an 
$n$-element set and Standard Young Tableaux
on $n$ boxes.
\item The length of the longest increasing subsequence of the involution
$v$ is equal
to the length of the first row of $RS(v)$, and
\item the length of the longest decreasing subsequence of the involution
$v$ is equal
to the length of the first column of $RS(v)$.
\end{enumerate}

Readers who want to deepen their knowledge of the Robinson-Schensted
correspondence should consult the book \cite{sagan} of Bruce Sagan.
The Robinson-Schensted
correspondence makes Theorem \ref{wilf} even more intriguing, since both
sides of (\ref{wilfeq}) can be interpreted in terms of Standard Young
Tableaux as well as in terms of permutations. 

In Section 3, we will also need the following, somewhat less well-known
result of Janet Simpson Beissinger, which is implicit in an earlier 
paper of Marcel-Paul Sch\"utzenberger \cite{schutz}.

\begin{theorem} \cite{beissinger} \label{beissinger}
Let $v$ be an involution, and let $RS(v)$ be its image under the 
Robinson-Schensted correspondence. Then the number of fixed points
of $v$ is equal to the number of odd columns of $RS(v)$.
\end{theorem}

\section{When $k=2n$}
In this section, we bijectively prove Theorem \ref{wilf} in the special
case when $k=2n$. It is clear that in that special case, the requirement
on the increasing subsequences on the left-hand side
of (\ref{wilfeq}), and the requirement on the length of rows on
the right-hand side of (\ref{wilfeq})  are automatically satisfied.
 Therefore,
if $y(m)$ denotes the number of involutions of an $m$-element set, then
 Theorem
\ref{wilf} simplifies to the following proposition.

\begin{proposition} \label{nolimits}
For all positive integers $n$, we have
\begin{equation}
\label{easyone} 
{2n\choose n} n! = \sum_{r=0}^{2n} {2n\choose r} (-1)^ry(r)y(2n-r).
\end{equation}
\end{proposition}

\begin{proof}
Let $[i]$ denote the set $\{1,2,\cdots ,i\}$.
 Let $A_n$ be the set of all permutations of the elements
of  $n$-element subsets of $[2n]$. Then the left-hand side of 
(\ref{easyone}) is equal to $|A_n|$.

Let $B_n$ be the set of ordered pairs $(p,q)$, where $p$ is an involution
on a subset $s_p$ of $[2n]$, and $q$ is an involution on the
set $[2n]-s_p$, the complement of $s_p$ in $[2n]$. 
Then the right-hand side of (\ref{easyone}) counts the elements of $B_n$
taking the parity of $r$ into account. More precisely, the 
 right-hand side of (\ref{easyone}) is equal to the number of 
elements of $B_n$ in which $|s_p|$ has an even size minus
the number of 
elements of $B_n$ in which $|s_p|$ has an odd size.

Now we are going to define an involution $f$ on a subset of $B_n$. Let
$(p,q)\in B_n$. As $p$ is an involution, all cycles of $p$ are of length one
(these are also called fixed points) or length two.  Let 
$F(p,q)$ be the set of all fixed points of $p$ and of all fixed points of $q$.
Let $M(p,q)$
 be the maximal element of $F(p,q)$ as long as $F(p,q)$ is a non-empty set.
 Now {\em move} $M(p,q)$
 to the {\em other}
involution in $(p,q)$. That is, if $M(p,q)$ was a fixed point of $p$, then 
move $M(p,q)$ to $q$, and if $M(p,q)$ was a fixed point of $q$, then move
 $M$ to $p$.
Call the resulting pair of involutions $f(p,q)=(p',q')$. 

\begin{example} Let $n=4$,  let $p=(31)(62)(5)$, and let $q=(7)(84)$.
Then $F(p,q)=\{5,7\}$, so $M(p,q)=7$, and therefore, $f(p,q)=(p',q')$, where
$p'=(31)(5)(62)(7)$ and $q'=(84)$.
\end{example}

It is clear that $f(p',q')=(p,q)$, since $F(p,q)=F(p',q')$, and
so $M(p,q)=M(p',q')$. So applying $f$ a second time simply moves
$M(p,q)$ back to its original place.

As the number of elements in $p$ and in $p'$ differs by exactly one, 
these two numbers are of different parity, and so the total contribution of
$(p,q)$ and $f(p,q)$ to the right-hand side of (\ref{easyone}) is 0.
Therefore, the only pairs $(p,q)$ whose contribution is not canceled
by the contribution of 
$f(p,q)$ are the pairs for which $f(p,q)$ is {\em not defined}, that is,
pairs $(p,q)$ in which both $p$ and $q$ are {\em fixed point-free 
involutions}. 

Noting that fixed point-free involutions are necessarily of even length, 
this shows that (\ref{easyone}) will be proved if we can show
that
\begin{equation} \label{fffree}
{2n\choose n} n! = \sum_{r=0}^{2n} {2n\choose r} x(r)x(2n-r),
\end{equation}
where $x(r)$ is the number of {\em fixed point-free} involutions
of length $r$.

This equality is straightforward to prove computationally, using the
fact that $x(2t)=(2t-1)\cdot (2t-3)\cdot \cdots \cdot 1=(2t-1)!!$ and
$x(2t+1)=0$. However,
for the sake of combinatorial purity, we provide a bijective proof. 
 
The left-hand side counts the ways to choose $n$ elements 
$a_1,a_2,\cdots ,a_n$ of $[2n]$ and then to arrange them in a line. 
Let $a\in A_n$ denote such an choice and arrangement. 
Now let $i_1<i_2<\cdots <i_n$ be the elements of $[2n]$ that we did {\em not}
choose, listed increasingly. Take the fixed point-free involution
whose cycles are the 2-cycles $(i_j,a_j)$, for $1\leq j\leq n$. Color
the cycles in which $i_j<a_j$ red, and the cycles in which $i_j>a_j$ blue.
Call the obtained fixed point-free permutation with bicolored cycles
$g(a)$.

It is then clear that $g$ maps into the set $D_n$ 
of fixed-point free permutations
on $[2n]$ whose cycles are colored red or blue. The right-hand side of 
(\ref{fffree}) counts precisely such involutions. Finally, it is 
straightforward to see that $g:A_n\rightarrow D_n$ is a bijection as
it has an inverse. (Just choose the smaller entry in each of the red cycles
and the larger entry in each of the blue cycles to recover $i_1,i_2,\cdots
,i_n$.) This completes the proof of (\ref{fffree}), and therefore, 
of Proposition \ref{nolimits}.
\end{proof}

\section{When $k$ is odd}

If we want to find a combinatorial proof of Theorem \ref{wilf} along the
line of the proof of Proposition \ref{nolimits}, we encounter several
difficulties. First, inserting a new fixed point into a partial permutation
can increase the length of its longest increasing subsequence, taking
it thereby out of the set that is being counted. More importantly, 
equality (\ref{fffree}) no longer holds if we replace $n!$ by $u_k(n)$ on
its left-hand side, and $x(h)$ by the number of fixed point-free involutions
with no increasing subsequences longer than $k$ on its right-hand side. 
Indeed, for $k=2$ and $n=3$, the left-hand side would be
${2n\choose n} u_2(3)=20\cdot 5=100$, while the right-hand side would
be $10+15\cdot 3+15\cdot 3 +10=110$.

It is  surprising that for the case of odd $k$, 
fixed points, and fixed point-free involutions, turn out to be relevant 
again. We point out that we will be considering
involutions without long {\em decreasing} rather than increasing subsequences.

Note that   $y_k(r)$ is equal to both the number of
 {\em involutions} on an $r$-element set
 with no {\em increasing} subsequences longer 
than $k$, and the number    {\em involutions} on an $r$-element set with no
 {\em decreasing subsequences} longer than $k$ (just take conjugates of 
the corresponding Standard Young Tableaux). However, this symmetry is
broken if we restrict our attention to {\em fixed point-free involutions},
since the conjugate of a tableaux with even columns only may have odd columns,
and our claim follows from Theorem \ref{beissinger}.

Let $x_k(r)$ be the number of {\em fixed point-free} involutions of
length $r$ with no {\em decreasing subsequences} with more than $k$ elements.
Note that $x_k(r)=0$ if $r$ is odd. 

\begin{theorem} \label{oddk}
For all positive integers $n$, and for all {\em odd} positive
integers $k$ the equality
\begin{equation} \label{tobeproved}
\sum_{r=0}^{2n} {2n\choose r} x_k(r)x_k(2n-r) = \sum_{r=0}^{2n}
 (-1)^r {2n\choose r} 
y_{k}(r)y_{k}(2n-r)\end{equation}
holds.
\end{theorem}

\begin{proof} Recall from the proof of Proposition \ref{nolimits} that
$B_n$ is the set of ordered pairs $(p,q)$, where $p$ is an involution
on a subset $s_p$ of $[2n]$, and $q$ is an involution on the
set $[2n]-s_p$, the complement of $s_p$ in $[2n]$.

Let $B(n,k,r)$ be the subset of $B_n$ consisting of 
pairs $(p,q)$ so that neither $p$ nor $q$ has
a decreasing subsequence longer than $k$. Note that here $p$ is an involution
of length $r$ and $q$ is an involution of length $2n-r$. 
It is then clear that 
\[ |B(n,k,r)| = {2n\choose r}  y_{k}(r)y_{k}(2n-r) .\]
Let $B(n,k)=\cup_r B(n,k,r)$.

Recall the involution $f$ from the proof of Proposition \ref{nolimits}, 
(the involution that took the largest fixed point present in $p\cup q$ and
moved it to the other involution), 
and let $f_{n,k}$ be the restriction of $f$ to the set $B(n,k)$.

Our theorem will be proved if we can show that $f_{n,k}$ maps into 
$B(n,k)$. Indeed, that would show that the only pairs $(p,q)\in B(n,k)$
whose contribution to the right-hand side of (\ref{tobeproved}) is not
canceled by the contribution of $f_{n,k}(p,q)$ are the pairs for which
$f(p,q)$ is not defined. It follows from the definition of $f_{n,k}$ that 
these are the pairs in which both $p$ and $q$ are fixed point-free 
involutions.

Our main tool is the following lemma.

\begin{lemma} \label{fixedpoint}
Let $w$ be an involution whose longest decreasing subsequence is of length
$2m+1$. Then each longest decreasing subsequence of $w$ must contain a fixed
point. 
\end{lemma}

\begin{proof} We use induction on $z$, the number of fixed points of $w$. 
If $z=0$, then the statement is vacuously true, since by Theorem 
\ref{beissinger} the
Standard Young Tableau corresponding to $w$ has no odd columns, so 
the length of its first column (and so, the length of the longest decreasing
subsequence of $w$) cannot be odd. 

Otherwise, assume that we know that the statement holds for $z-1$.
Also assume that $w$ has $z>0$ fixed points, and $w$ has a longest
decreasing subsequence $s$ of length $2m+1$ that does not contain any
fixed points. Remove a fixed point from $w$ to get $w'$. Then $w'$ still
has  a longest
decreasing subsequence $s$ of length $2m+1$ that contains no fixed points,
 even though $w'$ has only
$z-1$ fixed points, contradicting our induction hypothesis. 
\end{proof}

Let $(p,q)\in B(n,k)$. In order to show that  $f_{n,k}$ maps into 
$B(n,k)$, we need to show that $f_{n,k}(p,q)=f(p,q)=(p',q')$ has no decreasing
subsequence longer than $k$. The action of $f$ on $(p,q)$ consists of taking
a fixed point of one of $p$ and $q$ and adding it to the other. We can
assume without loss of generality that a fixed point of $p$ is being moved to
$q$. So the longest decreasing subsequence of $p'$ is not longer than that
of $p$, and so, not longer than $k$
 since $p'$ is a substring of $p$. There remains to show that the
longest decreasing subsequence of $q'$ is also not longer than $k$. 

As $q'$ differs from  $q$ only by the insertion of the fixed point $M=M(p,q)$,
 the only way $q'$ could possibly have a
decreasing subsequence longer than $k$ would be when $q$ itself has
a decreasing subsequence of length $k=2m+1$. In that case, by
Lemma \ref{fixedpoint}, all maximum-length decreasing subsequences of $q$
contain a fixed point. So when $M$ is inserted into $q$, and $q'$ is formed,
$M$ cannot extend any of the maximum-length decreasing subsequences of
 $q$ because that would mean that {\em two} fixed points are part of the 
same decreasing subsequence. That is impossible, since fixed points form
increasing subsequences.
  
So indeed, $f_{n,k}$ maps into $B(n,k)$, and our claim is proved. 
\end{proof}

\subsection{The special case $k=3$}

The first special case of Theorem \ref{oddk} is when $k=1$. Then
$x_k(r)=0$ for any $r$, while $y_k(r)=1$ for any $r$. So
(\ref{tobeproved}) simplifies to the well-known binomial-coefficient 
identity
\[0=\sum_{r=0}^{2n} (-1)^r {2n\choose r} .\]

The special case of $k=3$ is more interesting. We point out that
in this case, it is known \cite{gouyou}
 that $y_3(n)=\sum_{i=0}^{\lfloor n/2 \rfloor}
{n\choose 2i} C_i$, where $C_i={2i\choose i}/(n+1)$
 is the $i$th Catalan number.  The numbers $y_3(n)$ are  called
the {\em Motzkin numbers}.

It follows from Theorem \ref{beissinger} that if $v$ is fixed-point free,
then $RS(v)$ has no odd columns. Therefore, $x_{2m+1}(r)=x_{2m}(r)$.
In particular, for $k=3$, Theorem \ref{oddk} simplifies to 
\[\sum_{r=0}^{2n} (-1)^r {2n\choose r} 
y_{3}(r)y_{3}(2n-r) = \sum_{r=0}^{2n} {2n\choose r} x_2(r)x_2(2n-r).\]
Note that $x_2(r)$ is just the number of Standard Young Tableaux
in which each column is of length two (of even length not more than two).
The number of such tableaux is well-known (see for instance Exercise 6.19.ww
of \cite{stanley}) to be the Catalan number
$C_{r/2}$ if $r$ is even, and of course, 0 if $r$ is odd. 
Therefore, the previous displayed equation simplifies to 
\[\sum_{r=0}^{2n} (-1)^r {2n\choose r} 
y_{3}(r)y_{3}(2n-r) =  \sum_{i}  {2n \choose 2i} C_iC_{n-i}.\]

It turns out that the right-hand side is a well-known sequence. It is 
sequence A005568 in \cite{sloane}. In particular, it is proved in \cite{guy},
 that the $n$th element $f_n$ of this sequence has the closed form $f_n=
C_nC_{n+1}$. Furthermore,  it is shown in  
\cite{gouyou} that  $f_n=y_4(2n)$.

So we have proved the following identity.

\begin{corollary}
For all positive integers $n$, we have
\[\sum_{r=0}^{2n} (-1)^r {2n\choose r} 
y_{3}(r)y_{3}(2n-r)=y_4(2n)=C_{n}C_{n+1}=
\frac{{2n\choose n}{2n+2 \choose n+1}}{(n+1)(n+2)}.\]
\end{corollary}

\vskip 2 cm

\centerline{{\bf Acknowledgment}}

I am indebted to all three referees of this manuscript for their corrections
and elucidating comments.

\end{document}